\documentclass[12pt, reqno]{amsart}

\usepackage{amsmath, amssymb, amscd, amsthm, comment, amsxtra}
\usepackage{graphicx}
\usepackage{subfigure}

\usepackage[pdftex, linktocpage=true]{hyperref}
\usepackage{euscript}

\usepackage[format=plain,labelfont=bf,up,width=.99\textwidth]{caption}

\renewcommand{\arg}{\operatorname{arg}}

\theoremstyle{plain}
\newtheorem{thm}{Theorem}[section]
\newtheorem{lemma}[thm]{Lemma}
\newtheorem{lem}[thm]{Lemma}
\newtheorem{prop}[thm]{Proposition}
\newtheorem{remark}[thm]{Remark}
\newtheorem{cor}[thm]{Corollary}
\theoremstyle{definition}
\newtheorem{defn}[thm]{Definition}

\newcommand{\angdev}{\Delta_{\operatorname{arg}}}

\newcommand{\comdia}[8]
{
\begin{displaymath}
\begin{CD}
#1 @> #5 >> #2\\
@VV #7 V @VV #8 V\\
#3 @> #6 >> #4
\end{CD}
\end{displaymath}
}

\newcommand{\thmref}[1]{Theorem~\ref{#1}}

\newcommand{\C}{\mathbb{C}}

\newcommand{\D}{\mathbb{D}}

\newcommand{\eps}{\epsilon}

\newcommand{\cN}{{\mathcal N}}

\newcommand{\cC}{{\mathcal C}}
\newcommand{\cH}{{\mathcal H}}
\newcommand{\cR}{{\mathcal R}}

\newcommand{\cB}{{\mathcal B}}

\newcommand{\CC}{{\Bbb C}}

\newcommand{\DD}{{\Bbb D}}

\newcommand{\QQ}{{\Bbb Q}}

\title[Siegel disk boundaries]{The boundaries of golden-mean Siegel disks in the complex quadratic H{\'e}non family are not smooth}

\author{Michael Yampolsky}
\author{Jonguk Yang}

\begin{document}

\begin{abstract}
As was recently shown by the first author and others in \cite{GaRY}, golden-mean Siegel disks of sufficiently dissipative complex quadratic H{\'e}non maps are bounded by topological circles. In this paper we investigate the geometric properties of such curves, and demonstrate that they cannot be $C^1$-smooth.
\end{abstract}

\maketitle

\section{Introduction}
\label{sec:intro}
Up to a biholomorphic conjugacy, a complex quadratic H{\'e}non map can be written as
\[
H_{c,a}(x,y)=(x^{2}+c + ay,ax) \hspace{5mm} \text{for} \hspace{5mm} a\neq 0;
\]
this form is unique modulo the change of coordinates 
 $(x,y)\mapsto (x,-y)$, which conjugates $H_{c,a}$ with $H_{c,-a}$.   
In this paper we will always assume that the H{\'e}non map is dissipative, $|a|<1$. Note that for $a=0$, the map $H_{c,a}$ degenerates to
$$(x,y)\mapsto (f_c(x),0),$$ where $f_c(x)=x^2+c$ is a one-dimensional quadratic polynomial. Thus for a fixed small value of $a_0$, the one parameter family $H_{c,a_0}$ can be seen as a small perturbation of the quadratic family. 

As usual, we let $K^{\pm}$ be the sets of points that do not escape to infinity under forward, respectively backward iterations of the H{\'e}non map. Their topological boundaries are $J^{\pm}=\partial K^{\pm}$. Let $K=K^{+}\cap K^{-}$ and $J=J^{-}\cap J^{+}$. The sets $J^{\pm}$, $K^{\pm}$ are unbounded, connected sets in $\C^{2}$ (see \cite{BS1}). The sets $J$ and $K$ are compact (see \cite{HOV1}).  In analogy to one-dimensional dynamics, the set $J$ is called the Julia set of the H{\'e}non map. 

Note that  a  H{\'e}non map $H_{c,a}$ is determined by the multipliers $\mu$ and $\nu$ at a fixed point uniquely up to changing the sign of $a$. In particular,
$$\mu\nu=-a^2,$$
 the parameter $c$ is a function of $a^2$ and $\mu$:
$$c=(1-a^{2})\left(\frac{\mu}{2}-\frac{a^{2}}{2\mu}\right)-\left(\frac{\mu}{2}-\frac{a^{2}}{2\mu}\right)^{2}.$$
Hence, we sometimes  write $H_{\mu,\nu}$ instead of $H_{c,a}$, when convenient.

When $\nu=0$, the H{\'e}non map degenerates to 
\begin{equation}
\label{Plambda} H_{\mu,0}(x,y) = (P_{\mu}(x),0)\text{, where }P_{\mu}(x)=x^{2}+\mu/2-\mu^{2}/4.
\end{equation}

We say that a dissipative H{\'e}non map $H_{c,a}$ has a {\it semi-Siegel fixed point}  (or simply that $H_{c,a}$ is semi-Siegel) if the  eigenvalues of the linear part of $H_{c,a}$ at that fixed point are $\mu=e^{2\pi i \theta}$, with $\theta\in(0,1)\setminus \QQ$ and $\nu$, with $|\nu|<1$, and $H_{c,a}$ is locally biholomorphically conjugate to the linear map 
$$L(x,y)=(\mu x,\nu y).$$
The classic theorem of Siegel states, in particular, that $H_{\mu,\nu}$ is semi-Siegel whenever $\theta$ is Diophantine, that is $q_{n+1}<cq_n^d$, where $p_n/q_n$ are the continued fraction convergents of $\theta$. 
The existence of a linearization is a local result, however, 
in this case there exists a linearizing biholomorphism $\phi:\D\times \C\rightarrow \C^{2}$ sending $(0,0)$ to the semi-Siegel fixed point,
$$H_{\mu,\nu}\circ \phi=\phi\circ L,$$
such that the image $ \phi(\D\times\C)$ is {\it maximal} (see \cite{MNTU}).
 We call  $ \phi(\D\times\C)$ the {\it Siegel cylinder}; it is a connected component of the interior of $K^+$ and its boundary coincides with $J^+$ (see \cite{BS2}). We let 
$$\Delta=\phi(\D\times\{0\}),$$
and by analogy with the one-dimensional case call it the {\it Siegel disk} of the H{\'e}non map. Clearly, the Siegel cylinder is 
equal to the stable manifold $W^s(\Delta)$, and 
  $\Delta\subset K$ (which is always bounded). 
Moreover,  $\partial \Delta\subset J$, the Julia set of the H{\'e}non map.

\begin{remark}
Let $\textbf{q}$ be the semi-Siegel fixed point of the H{\'e}non map. Then $\Delta\subset W^{c}(\textbf{q})$, the center manifold of $\textbf{q}$ (see e.g. \cite{S} for a definition of $W^c$). The center manifold is not unique in general, but all center manifolds $W^{c}(\textbf{q})$ coincide on the Siegel disk. This phenomenon is nicely
illustrated in \cite{O}, Figure 5.
\end{remark}

In a recent paper \cite{GaRY} it was shown that:

\begin{thm}[\cite{GaRY}]
\label{thm:gry}
 There exists $\delta>0$ such that the following holds.
Let $\theta_*=(\sqrt{5}-1)/2$ be the inverse golden mean, $\mu_*=e^{2\pi i\theta_*}$, and let $|\nu|<\delta$. Then the boundary of the Siegel disk of $H_{\mu_*,\nu}$ is a homeomorphic image of the circle.

Furthermore, the linearizing map
\begin{equation}
\label{eq:lin}
\phi:\D\times \{0\}\rightarrow \Delta
\end{equation}
extends continuously and injectively to the boundary. However, 
the restriction  $$\phi:S^1\times \{0\}\rightarrow \partial \Delta$$ is not $C^1$-smooth.
\end{thm}

\noindent
This is the first  result of its kind on the structure of the boundaries of Siegel disks of complex H{\'e}non maps. It is based on  a renormalization theory for two-dimensional dissipative H{\'e}non-like maps, developed  in \cite{GaYa2}. 
Below, we will briefly review the relevant renormalization results.

\thmref{thm:gry} raises a natural question whether the boundary $\partial \Delta$ can ever lie on a smooth curve.
Calssical results (see \cite{War}) imply that the smoothness of $\partial \Delta$ must be less than $C^{1+\eps}$ -- otherwise, $\phi$ would have a $C^{1+\eps}$ extension to the boundary, contradicting Theorem~\ref{thm:gry}. However, we can ask, whether $\partial\Delta$ can be a $C^1$-smooth curve. In the present note we answer this in the negative:

\medskip
\noindent
{\bf Main Theorem.} {\it Let $\delta>0$ be as in \thmref{thm:gry} and $|\nu|<\delta$. Then the boundary of the Siegel disk of $H_{\mu_*,\nu}$ is not $C^1$-smooth.  }

\medskip
\noindent
 \section{Review of renormalization theory for Siegel disks}
In this section we give a brief summary of the relevant statements on renormalization of Siegel disks; we refer the reader to \cite{GaYa2} for the details. 
\subsection{One-dimensional renormalization: almost-commuting pairs}
For a domain $Z \subset \mathbb{C}$, we denote $\mathcal{H}(Z)$ the Banach space of bounded analytic functions $f : Z \to \mathbb{C}$ equipped with the norm
\begin{equation}\label{eq:snorm1}
\| f\|= \sup_{x \in Z}|f(x)|.
\end{equation}

Denote $\mathcal{H}(Z, W)$ the Banach space of bounded pairs of analytic functions $\zeta=(f, g)$  from domains $Z \subset \mathbb{C}$ and $W \subset \mathbb{C}$ respectively to $\CC$ equipped with the norm
\begin{equation}
\label{eq:unorm1}\| \zeta\|= \frac{1}{2} \left(   \|f\|+  \|g\| \right).
\end{equation}
Henceforth, we assume that the domains $Z$ and $W$ contain $0$.

For a pair $\zeta=(f,g)$, define the \emph{rescaling map} as
\begin{equation}\label{eq:rescaling1}
\Lambda(\zeta) := (s_\zeta^{-1} \circ f \circ s_\zeta, s_\zeta^{-1} \circ g \circ s_\zeta),
\end{equation}
where
\begin{displaymath}
s_\zeta(x) := \lambda_\zeta x
\hspace{5mm} \text{and} \hspace{5mm}
\lambda_\zeta := g(0).
\end{displaymath}

\begin{defn}
We say that $\zeta= (\eta, \xi) \in \mathcal{H}(Z,W)$ is a {\it critical pair} if $\eta$ and $\xi$ have a simple unique critical point at $0$. The space of critical pairs is denoted by $\cC(Z,W)$.
\end{defn}

\begin{defn}
We say that $\zeta= (\eta, \xi) \in \mathcal{C}(Z,W)$ is a {\it commuting pair} if
\begin{displaymath}
\eta \circ \xi = \xi \circ \eta.
\end{displaymath}
\end{defn}

\begin{defn}
We say that $\zeta= (\eta, \xi) \in \mathcal{C}(Z,W)$ is an {\it almost commuting pair} (cf. \cite{Bur,Stir}) if
\begin{displaymath}
\frac{d^i (\eta \circ \xi - \xi \circ \eta)}{dx^i}(0) = 0
\hspace{5mm} \text{for} \hspace{5mm}
i = 0, 2,
\end{displaymath}
and
\begin{displaymath}
\xi(0)=1.
\end{displaymath}
The space of almost commuting pairs is denoted by $\cB(Z, W)$.
\end{defn}

\begin{prop}[cf. \cite{GaYa2}]
The spaces $\cC(Z,W)$ and $\cB(Z,W)$ have the structure of an immersed Banach submanifold of $\cH(Z,W)$ of codimension $2$ and $5$ respectively.
\end{prop}

Denote
\begin{displaymath}
c(x) := \bar{x}.
\end{displaymath}

\begin{defn} \label{def:1d pre-renormalization}
Let $\zeta = (\eta, \xi) \in \cB(Z,W)$. The \emph{pre-renormalization} of $\zeta$ is defined as:
\begin{displaymath}
p\mathcal{R}((\eta, \xi)) := (\eta \circ \xi, \eta).
\end{displaymath}
The \emph{renormalization} of $\zeta$ is defined as:
\begin{displaymath}
\mathcal{R}((\eta, \xi)):= \Lambda(c \circ \eta \circ \xi\circ c, c \circ \eta \circ c).
\end{displaymath}
\end{defn}

It is easy to see that
\begin{prop}
The renormalization of an (almost) commuting pair is an (almost) commuting pair (on different domains).
\end{prop}

It is convenient to introduce the following multi-index notation. Let $\mathcal{I}$ be the space of all finite multi-indexes
\begin{displaymath}
\overline{\omega} = (a_0, \ldots{}, a_n) \in (\{0\} \cup \mathbb{N})^n
\hspace{5mm} \text{for some } n \in \mathbb{N},
\end{displaymath}
with the partial ordering relation defined as follows:
\begin{displaymath}
(a_0, a_1, \ldots{}, a_k, b) \prec (a_0, a_1, \ldots{}, a_n, a_{n+1})
\end{displaymath}
if either $k < n$ and $b \leq a_{k+1}$, or $k = n$ and $b < a_{n+1}$. For a pair $\zeta = (\eta, \xi)$ and a multi-index $\overline{\omega} = (a_0, \ldots{}, a_n) \in \mathcal{I}$, denote
\begin{displaymath}
\zeta^{\overline{\omega}} = \phi^{a_n} \circ \ldots{} \circ \xi^{a_1} \circ \eta^{a_0}
\end{displaymath}
where $\phi$ is either $\eta$ or $\xi$, depending on whether $n$ is even or odd. Define a sequence $\{\overline{\alpha}_0, \overline{\alpha}_1, \ldots\} \subset \mathcal{I}$ such that
\begin{equation} \label{renormalization iterate}
p\mathcal{R}^n(\zeta) = (\zeta^{\overline{\alpha}_n}, \zeta^{\overline{\alpha}_{n-1}}).
\end{equation}

The following is shown in \cite{GaYa2}:
\begin{thm} \label{1d renormalization}
There exist topological disks $\hat Z \Supset Z$ and $\hat W \Supset W$, and an almost commuting pair $\zeta_* = (\eta_*, \xi_*) \in \cB(Z,W)$ such that the following holds:
\begin{enumerate}
\item There exists a neighbourhood $\cN(\zeta_*)$ of $\zeta_*$ in the submanifold $\cB(Z,W)$ such that
\begin{displaymath}
\cR : \cN(\zeta_*) \to \cB(\hat Z,\hat W)
\end{displaymath}
is an anti-analytic operator.
\item The pair $\zeta_*$ is the unique fixed point of $\cR$ in $\cN(\zeta_*)$.
\item The differential $D\cR^2|_{\zeta_*}$ is a compact linear operator. It has a single, simple eigenvalue with modulus greater than $1$. The rest of its spectrum lies inside the open unit disk $\DD$  (and hence is compactly contained in $\DD$ by the spectral theory of compact operators).
\end{enumerate}
\end{thm}

\subsection{Renormalization of two-dimensional maps}
For a domain $\Omega \subset \mathbb{C}^2$, we denote $O(\Omega)$ the Banach space of bounded analytic functions $F : \Omega \to \mathbb{C}^2$ equipped with the norm
\begin{equation}\label{eq:snorm2}
\| F\|= \sup_{(x,y) \in \Omega}|F(x,y)|.
\end{equation}
Define
\begin{equation}\label{eq:ynorm1}
\| F\|_y := \sup_{(x,y) \in \Omega}|\partial_y F(x,y)|.
\end{equation}
Moreover, for
\begin{displaymath}
F =
\begin{bmatrix}
f_1 \\
f_2
\end{bmatrix},
\end{displaymath}
define
\begin{equation}\label{eq:diagnorm1}
\| F\|_{\text{diag}} := \sup_{(x,y) \in \Omega}|f_1(x,y)-f_2(x,y)|.
\end{equation}

Denote $O(\Omega, \Gamma)$ the Banach space of bounded pairs of analytic functions $\Sigma=(F,G)$  from domains $\Omega \subset \mathbb{C}^2$ and $\Gamma \subset \mathbb{C}^2$ respectively to $\CC^2$ equipped with the norm
\begin{equation}\label{eq:unorm2}
\| \Sigma\|= \frac{1}{2} \left(\|F\|+  \|G\|\right).
\end{equation}
Define
\begin{equation}\label{eq:ynorm2}
\| \Sigma\|_y :=  \frac{1}{2} \left(\|F\|_y +  \|G\|_y \right).
\end{equation}
Moreover,
\begin{equation}\label{eq:diagnorm}
\| \Sigma\|_{\text{diag}} := \frac{1}{2} \left(\|F\|_{\text{diag}}+  \|G\|_{\text{diag}} \right).
\end{equation}

Henceforth, we assume that
\begin{displaymath}
\Omega=Z  \times Z \hspace{5mm} \text{and} \hspace{5mm} \Gamma=W \times W,
\end{displaymath}
where $Z$ and $W$ are subdomains of $\CC$ containing $0$. For a function
\begin{displaymath}
F(x,y):=
\begin{bmatrix}
f_1(x,y)\\
f_2(x,y)
\end{bmatrix}
\end{displaymath}
from $\Omega$ or $\Gamma$ to $\mathbb{C}^2$, we denote
\begin{displaymath}
 p_1F(x):=f_1(x,0) \hspace{5mm} \text{and} \hspace{5mm}  p_2F(x):=f_2(x,0).
\end{displaymath}
For a pair $\Sigma = (F, G)$, define the \emph{rescaling map} as
\begin{equation} \label{eq:rescaling2}
\Lambda(\Sigma) := (s_\Sigma^{-1} \circ F \circ s_\Sigma, s_\Sigma^{-1} \circ G \circ s_\Sigma),
\end{equation}
where
\begin{displaymath}
s_\Sigma(x,y) := (\lambda_\Sigma x, \lambda_\Sigma y)
\hspace{5mm} \text{and} \hspace{5mm}
\lambda_\Sigma :=  p_1 G(0).
\end{displaymath}

\begin{defn}
For $0<\kappa \leq \infty$, we say that $\Sigma \in O(\Omega, \Gamma)$ is a {\it $\kappa$-critical pair} if $ p_1 A$ and $ p_1 B$ have a simple unique critical point which is contained in a $\kappa$-neighbourhood of $0$. The space of $\kappa$-critical pairs in $O(\Omega, \Gamma)$ is denoted by $\cC_2(\Omega,\Gamma, \kappa)$.
\end{defn}

\begin{defn}
We say that $\Sigma=(A,B) \in \mathcal{C}_2(\Omega, \Gamma, \kappa)$ is a \emph{commuting pair} if
\begin{displaymath}
A \circ B = B \circ A.
\end{displaymath}
\end{defn}

\begin{defn}
We say that $\Sigma= (A, B) \in \mathcal{C}_2(\Omega,\Gamma, \kappa)$ is an {\it almost commuting pair} if
\begin{displaymath}
\frac{d^i  p_1[A, B]}{dx^i}(0) := \frac{d^i  p_1(A \circ B - B \circ A)}{dx^i}(0) = 0
\hspace{5mm} \text{for} \hspace{5mm}
i = 0, 2,
\end{displaymath}
and
\begin{displaymath}
 p_1 B(0)=1.
\end{displaymath}
The space of almost commuting pairs in $\mathcal{C}_2(\Omega,\Gamma, \kappa)$ is denoted by $\cB_2(\Omega,\Gamma, \kappa)$.
\end{defn}

\begin{prop}[cf.\cite{GaYa2}]
The space $\cB_2(\Omega,\Gamma, \kappa)$ has the structure of an immersed Banach submanifold of $O(\Omega,\Gamma)$ of codimension $3$.
\end{prop}

For $0< \epsilon, \delta \leq \infty$, let $O(\Omega, \Gamma, \epsilon, \delta)$ be the open subset of $O(\Omega, \Gamma)$ consisting of pairs $\Sigma=(A,B)$ such that the following holds:
\begin{enumerate}
\item $\|\Sigma\|_y < \epsilon$, and
\item $\|\Sigma\|_{\text{diag}} < \delta$.
\end{enumerate}
We denote
\begin{equation}\label{2d crit pair space}
\mathcal{C}_2(\Omega, \Gamma, \epsilon, \delta, \kappa) := O(\Omega, \Gamma, \epsilon, \delta) \cap \mathcal{C}_2(\Omega, \Gamma, \kappa), 
\end{equation}
and
\begin{equation}\label{2d pair space}
\mathcal{B}_2(\Omega, \Gamma, \epsilon, \delta, \kappa) := O(\Omega, \Gamma, \epsilon, \delta) \cap \mathcal{B}_2(\Omega, \Gamma,\kappa).
\end{equation}

\begin{prop} [cf. \cite{GaYa2}] \label{renormalization projection}
If $\epsilon$, $\delta$, and $\kappa$ are sufficiently small, then there exists an analytic map $\Pi_{\text{\emph{ac}}} : C_2(\Omega, \Gamma, \epsilon, \delta, \kappa) \to \mathcal{B}_2(\Omega, \Gamma, \epsilon, \delta, \kappa)$ such that
\begin{equation}
\Pi_{\text{\emph{ac}}}|_{\mathcal{B}_2(\Omega, \Gamma, \epsilon, \delta, \kappa)} \equiv \text{ \emph{Id}}.
\end{equation}
\end{prop}

\begin{lemma} \label{4th 1d renormalization}
Consider the sequence of multi-indexes $\{\overline{\alpha}_0, \overline{\alpha}_1, \ldots\} \subset \mathcal{I}$ defined by \eqref{renormalization iterate}. Let $\zeta_0 = (\eta_0, \xi_0) \in \mathcal{B}(Z,W)$ be a four times 1D renormalizable pair. There exists a neighbourhood $\mathcal{N}(\zeta_0) \subset \mathcal{H}(Z, W)$ of $\zeta_0$ such that for any pair $\zeta = (\eta, \xi)$ in $\mathcal{N}(\zeta_0)$, the pair
\begin{displaymath}
\mathcal{R}^4(\zeta) = \Lambda(p\mathcal{R}^4(\zeta)) := (s_{p\mathcal{R}^4(\zeta)}^{-1} \circ \zeta^{\overline{\alpha}_4} \circ s_{p\mathcal{R}^4(\zeta)}, s_{p\mathcal{R}^4(\zeta)}^{-1} \circ \zeta^{\overline{\alpha}_3} \circ s_{p\mathcal{R}^4(\zeta)}),
\end{displaymath}
is a well defined element of $\mathcal{H}(Z, W)$.
\end{lemma}

It is instructive to note that
\begin{displaymath}
p\mathcal{R}^4(\zeta) = (\zeta^{\overline{\alpha}_4}, \zeta^{\overline{\alpha}_3}) = (\eta \circ \xi \circ \eta^2 \circ \xi \circ \eta \circ \xi \circ \eta, \eta \circ \xi \circ \eta^2 \circ \xi)
\end{displaymath}

Let $\mathcal{D}(\Omega, \Gamma, 0)$ be the subset of $O(\Omega, \Gamma)$ consisting of pairs $\Sigma=(A,B)$ such that the following holds:
\begin{enumerate}
\item The functions $A: \Omega \to \mathbb{C}^2$ and $B : \Gamma \to \mathbb{C}^2$ are of the form
\begin{displaymath}
A(x,y) =
\begin{bmatrix}
\eta(x) \\
h(x)
\end{bmatrix}
\hspace{5 mm} \text{and} \hspace{5mm}
B(x,y) =
\begin{bmatrix}
\xi(x) \\
g(x)
\end{bmatrix}.
\end{displaymath}
\item The pair $\zeta := (\eta, \xi)$ is contained in $\mathcal{B}(Z, W)$ and is four-times 1D renormalizable.
\item The function $g$ is conformal on $\eta^2 \circ \xi \circ \eta \circ \xi \circ \eta(U)$ and $\eta^2 \circ \xi(U)$, where
\begin{displaymath}
U:= \lambda_{p\mathcal{R}^4(\zeta)} Z \cup W.
\end{displaymath}
\end{enumerate}

Let $\mathcal{D}(\Omega, \Gamma, \epsilon) \subset O(\Omega, \Gamma, \epsilon, \infty)$ be a neighbourhood of $\mathcal{D}(\Omega, \Gamma, 0)$ consisting of pairs $\Sigma = (A,B)$ such that $\Lambda(\tilde{\Sigma})$, where
\begin{equation} \label{eq:4th renormalization}
\tilde{\Sigma}:= (\Sigma^{\overline{\alpha}_4}, \Sigma^{\overline{\alpha}_3}),
\end{equation}
is a well-defined element of $O(\Omega, \Gamma)$, and for $V := \lambda_{\tilde{\Sigma}} \Omega \cup \Gamma$:
\begin{enumerate}
\item $ p_1 A$ is conformal on $( p_1 A)^{-1}(V)$,
\item $ p_1 A \circ B$ is conformal on  $( p_1 A \circ B)^{-1}(V)$, and
\item $ p_2 B$ is conformal on $ p_1 A^2 \circ B \circ A \circ B \circ A(V)$ and $ p_ 1 A^2 \circ B(V)$.
\end{enumerate}

We define an isometric embedding $\iota$ of the space $\mathcal{H}(Z)$ to $\mathcal{O}(\Omega)$ as follows:
\begin{equation}
\label{eq:embed}
\iota(f)(x,y) = \iota(f)(x) :=
\begin{bmatrix}
f(x) \\
f(x)
\end{bmatrix}.
\end{equation}
We extend this definition to an isometric embedding of $\cH(Z,W)$ into $O(\Omega,\Gamma)$ as follows:
\begin{equation} \label{eq:embed pair}
\iota((\eta, \xi)) := (\iota(\eta), \iota(\xi)).
\end{equation}
Note that
\begin{displaymath}
\iota(\mathcal{B}(Z, W))=\mathcal{B}_2(\Omega, \Gamma, 0, 0, 0).
\end{displaymath}

Consider the fixed point $\zeta_*=(\eta_*, \xi_*) \in \cB(Z,W)$ of the 1D renormalization operator $\mathcal{R}$ given in \thmref{thm:gry}. Fix $\epsilon>0$, and let $\widehat{\mathcal{N}}(\iota(\zeta_*)) \subset \mathcal{D}(\Omega, \Gamma, \epsilon)$ be a neighbourhood of $\iota(\zeta_*)$ whose closure is contained in $\mathcal{D}(\Omega, \Gamma, \epsilon)$.

Let
\begin{displaymath}
\Sigma = (A, B)=
\left(
\begin{bmatrix}
a \\
h
\end{bmatrix},
\begin{bmatrix}
b \\
g
\end{bmatrix}
\right)
\end{displaymath}
be a pair contained in $\widehat{\mathcal{N}}(\iota(\zeta_*))$. Denote
\begin{displaymath}
\eta_i(x) :=  p_i A(x) \hspace{5mm} \text{and} \hspace{5mm} \xi_i(x) :=  p_i B(x) \hspace{5mm} , \hspace{5mm} \text{for } i \in \{1, 2\},
\end{displaymath}
and let
\begin{displaymath}
\zeta := (\eta_1, \xi_1).
\end{displaymath}

Denote
\begin{displaymath}
a_y(x) := a(x,y),
\end{displaymath}
and consider the following non-linear changes of coordinates:
\begin{equation} \label{eq:nonlinear change of coord}
H(x,y):=
\begin{bmatrix}
a_y(x) \\
y
\end{bmatrix}
\hspace{5mm} \text{and} \hspace{5mm}
V(x,y):=
\begin{bmatrix}
x \\
\eta_1 \circ \xi_1 \circ \xi_2^{-1}(y)
\end{bmatrix}.
\end{equation}

Observe that
\begin{displaymath}
A \circ H^{-1}(x,y) =
\begin{bmatrix}
a_y \circ a_y^{-1}(x) \\
g(a_y^{-1}(x),y)
\end{bmatrix} =
\begin{bmatrix}
x \\
g(a_y^{-1}(x),y)
\end{bmatrix}.
\end{displaymath}
Furthermore,
\begin{displaymath}
V \circ H \circ B =
\begin{bmatrix}
a_{g} \circ b \\
\eta_1 \circ \xi_1 \circ \xi_2^{-1} \circ g
\end{bmatrix}.
\end{displaymath}
Thus, we have
\begin{displaymath}
\|A \circ H^{-1} \|_y < O(\epsilon)
\hspace{5mm} \text{and} \hspace{5mm}
\|V \circ H \circ B - \iota(\eta_1 \circ \xi_1)\| < O(\epsilon)
\end{displaymath}
where defined. 

Let
\begin{displaymath}
A_1 := V \circ H \circ A^{-1} \circ \Sigma^{\overline{\alpha}_4} \circ A \circ H^{-1} \circ V^{-1},
\end{displaymath}
and
\begin{displaymath}
B_1 := V \circ H \circ A^{-1} \circ \Sigma^{\overline{\alpha}_3} \circ A \circ H^{-1} \circ V^{-1}.
\end{displaymath}
Define the \emph{pre-renormalization} of $\Sigma$ as
\begin{equation} \label{eq:pre-renormalization}
p\mathbf{R}(\Sigma)=\Sigma_1 := (A_1, B_1).
\end{equation}
By the definition of $\mathcal{D}(\Omega, \Gamma, \epsilon)$, the pair $p\mathbf{R}$ is a well-defined element of $O(\lambda_{\Sigma_1}\Omega, \lambda_{\Sigma_1}\Gamma)$. From the above inequalities, it follows that
\begin{equation} \label{eq:estimates}
\|p\mathbf{R}(\Sigma) - \iota(p\mathcal{R}^4(\zeta))\| < O(\epsilon)
\hspace{5mm} \text{and} \hspace{5mm}
\|p\mathbf{R}(\Sigma)\|_y < O(\epsilon^2).
\end{equation}

By the argument principle, if $\epsilon$ is sufficiently small, then the function $ p_1 B_1 \circ A_1$ has a simple unique critical point $c_a$ near $0$. Set
\begin{equation}\label{eq:crit projection1}
T_a(x,y):= (x + c_a, y),
\end{equation}
Likewise, the function $ p_1 T_a^{-1} \circ A_1 \circ B_1 \circ T_a$ has a simple unique critical point $c_b$ near $0$. Set
\begin{equation}\label{eq:crit projection2}
T_b(x,y):= (x + c_b, y).
\end{equation}
Note that if $\Sigma$ is a commuting pair (i.e. $A \circ B = B \circ A$), then $T_b \equiv$ Id.

Define the \emph{critical projection} of $p\mathbf{R}(\Sigma)$ as
\begin{equation} \label{eq:crit projection}
\Pi_{\text{crit}} \circ p\mathbf{R}(\Sigma) = (A_2, B_2) := (T_b^{-1} \circ T_a^{-1} \circ A_1 \circ T_a, T_a^{-1} \circ B_1 \circ T_a \circ T_b).
\end{equation}
Note that
\begin{displaymath}
0 =  p_1(B_2 \circ A_2)'(0) = ( p_1 A_2)'(0) + O(\epsilon^2),
\end{displaymath}
and likewise
\begin{displaymath}
0 =  p_1(A_2 \circ B_2)'(0) = ( p_1 B_2)'(0) + O(\epsilon^2).
\end{displaymath}
Hence,
\begin{equation} \label{eq:crit estimate}
( p_1 A_2)'(0) = O(\epsilon^2)
\hspace{5mm} \text{and} \hspace{5mm}
( p_1 B_2)'(0) = O(\epsilon^2).
\end{equation}
It follows that there exists a uniform constant $C>0$ such that the rescaled pair $\Lambda \circ \Pi_{\text{crit}} \circ p\mathbf{R}(\Sigma)$ is contained in $\mathcal{C}_2(\Omega, \Gamma, C\epsilon^2, C\epsilon, C\epsilon^2)$ (recall that this means $\Lambda \circ \Pi_{\text{crit}} \circ p\mathbf{R}(\Sigma)$ is a $C\epsilon^2$-critical pair with $C\epsilon^2$ dependence on $y$ that is $C\epsilon$ away from the diagonal; see \eqref{2d crit pair space}). 

Finally, define the \emph{2D renormalization} of $\Sigma$ as
\begin{equation}
\mathbf{R}(\Sigma) := \Pi_{\text{ac}} \circ \Lambda \circ \Pi_{\text{crit}} \circ p\mathbf{R}(\Sigma),
\end{equation}
where the projection map $\Pi_{\text{ac}}$ is given in proposition \ref{renormalization projection}.

\begin{prop} \label{commuting}
If $\Sigma = (A,B) \in \mathcal{D}(\Omega, \Gamma, \epsilon)$ is a commuting pair (i.e. $A \circ B = B \circ A$), then $\mathbf{R}(\Sigma)$ is a conjugate of $(\Sigma^{\overline{\alpha}_4}, \Sigma^{\overline{\alpha}_3})$.
\end{prop}

\begin{thm}
\label{thm:hyperb2}
Let $\zeta_*$ be the fixed point of the 1D renormalization given in \thmref{thm:gry}. For any sufficiently small $\epsilon >0$, let $\widehat{\mathcal{N}}(\iota(\zeta_*)) \subset \mathcal{D}(\Omega, \Gamma, \epsilon)$ be a neighbourhood of $\iota(\zeta_*)$ whose closure is contained in $\mathcal{D}(\Omega, \Gamma, \epsilon)$ . Then there exists a uniform constant $C>0$ depending on $\widehat{\mathcal{N}}(\iota(\zeta_*))$ such that the 2D renormalization operator
$$\mathbf{R} : \mathcal{D}(\Omega, \Gamma, \epsilon) \to O(\Omega, \Gamma),$$
is a well-defined compact analytic operator satisfying the following properties:

\begin{enumerate}
\item$\mathbf{R}|_{\widehat{\mathcal{N}}(\iota(\zeta_*))} : \widehat{\mathcal{N}}(\iota(\zeta_*)) \to \mathcal{B}_2(\Omega, \Gamma, C\epsilon^2, C\epsilon, C\epsilon^2)$.

\item If $\Sigma = (A, B) \in \widehat{\mathcal{N}}(\iota(\zeta_*))$ and $\zeta := ( p_1 A,  p_1 B)$, then
\begin{displaymath}
\|\mathbf{R}(\Sigma) - \iota(\mathcal{R}^4(\zeta))\| < C \epsilon.
\end{displaymath}
Consequently, if $\cN(\zeta_*) \subset \mathcal{B}(Z, W)$ is a neighbourhood of $\zeta_*$ such that $\iota(\cN(\zeta_*)) \subset \widehat{\mathcal{N}}(\iota(\zeta_*))$, then
\begin{displaymath}
\mathbf{R} \circ \iota|_{\cN(\zeta_*)} \equiv \iota \circ \mathcal{R}^4|_{\cN(\zeta_*)}.
\end{displaymath}

\item The pair $\iota(\zeta_*)$ is the unique fixed point of $\mathbf{R}$ in $\widehat{\mathcal{N}}(\iota(\zeta_*))$.

\item The differential $D_{\iota(\zeta_*)}\mathbf{R}$ is a compact linear operator whose spectrum coincides with that of  $D_{\zeta_*}\mathcal{R}^4$. More precisely, in the spectral decomposition of  $D_{\iota(\zeta_*)}\mathbf{R}$, the complement to the tangent space $T_{\iota(\zeta_*)}(\iota(\cN(\zeta_*)))$ corresponds to the zero eigenvalue.
\end{enumerate}
\end{thm}


We denote the stable manifold of the fixed point $\iota(\zeta_*)$ for the 2D renormalization operator $\mathbf{R}$ by $W^s(\iota(\zeta_*)) \subset \mathcal{D}(\Omega, \Gamma, \epsilon)$.

Let $H_{\mu_*,\nu}$ be the H{\'e}non map with a semi-Siegel fixed point $\mathbf{q}$ of multipliers $\mu_* = e^{2\pi i \theta_*}$ and $\nu$, where $\theta_* = (\sqrt{5}-1)/2$ is the inverse golden mean rotation number, and $|\nu|<\epsilon$. We identify $H_{\mu_*, \nu}$ as a pair in $\mathcal{D}(\Omega, \Gamma, \epsilon)$ as follows:
\begin{equation} \label{henonpair}
\Sigma_{H_{\mu_*,\nu}} := \Lambda(H_{\mu_*,\nu}^2, H_{\mu_*,\nu}).
\end{equation}

The following is shown in \cite{GaRY}:

\begin{thm} \label{thm:henon}
The pair $\Sigma_{H_{\mu_*,\nu}}$ is contained in  the stable manifold $W^s(\iota(\zeta_*)) \subset \mathcal{D}(\Omega, \Gamma, \epsilon)$ of the fixed point $\iota(\zeta_*)$ for the 2D renormalization operator $\mathbf{R}$.
\end{thm}


\section{Proof of Main Theorem}
\subsection{Preliminaries}
Let
\begin{displaymath}
\zeta_* = (\eta_*, \xi_*)
\end{displaymath}
be the fixed point of the 1D renormalization operator $\mathcal{R}$ given in theorem \ref{thm:gry}. By theorem \ref{thm:hyperb2}, the fixed point of the 2D renormalization operator
\begin{displaymath}
\mathbf{R} :  \widehat{\mathcal{N}}(\iota(\zeta_*)) \to \mathcal{B}_2(\Omega, \Gamma, C\epsilon^2, C\epsilon).
\end{displaymath}
is the diagonal embedding $\iota(\zeta_*)$ of $\zeta_*$. Thus, we have
\begin{displaymath}
\iota(\zeta_*)=\mathbf{R}(\iota(\zeta_*)) = (s_*^{-1} \circ \iota(\zeta)^{\overline{\alpha}_4} \circ s_*, s_*^{-1} \circ \iota(\zeta)^{\overline{\alpha}_3} \circ s_*),
\end{displaymath}
where
\begin{displaymath}
s_*(x,y) := (\lambda_* x, \lambda_* y)
\hspace{5mm} , \hspace{5mm}
|\lambda_*| < 1.
\end{displaymath}

Let $\Sigma = (A, B)$ be a pair contained in the stable manifold $W^s(\iota(\zeta_*))$ of the fixed point $\iota(\zeta_*)$. Assume that $\Sigma$ is commuting, so that
\begin{displaymath}
A \circ B = B \circ A.
\end{displaymath}
Set
\begin{displaymath}
\Sigma_n = (A_n, B_n) =
\left(
\begin{bmatrix}
a_n \\
h_n
\end{bmatrix}
,
\begin{bmatrix}
b_n \\
g_n
\end{bmatrix}
\right) :=
\mathbf{R}^n(\Sigma).
\end{displaymath}
Let
\begin{displaymath}
\eta_n(x) :=  p_1 A_n(x) = a_n(x, 0)
\hspace{5mm} \text{and} \hspace{5mm}
\xi_n(x) :=  p_1 B_n(x) = b_n(x,0).
\end{displaymath}
By theorem \ref{thm:hyperb2}, we may express
\begin{equation}\label{shadowing}
A_n = \iota(\eta_n) + E_n
\hspace{5mm} \text{and} \hspace{5mm}
B_n = \iota(\xi_n) + F_n
\end{equation}
where the error terms $E_n$ and $F_n$ satisfy
\begin{equation}\label{eq:error estimate}
\|E_n\| < C \epsilon^{2^{n-1}}
\hspace{5mm} \text{and} \hspace{5mm}
\|F_n\| < C \epsilon^{2^{n-1}}.
\end{equation}
Hence, the sequence of pairs $\{\Sigma_n\}_{n=0}^\infty$ converges to $\mathcal{B}_2(\Omega, \Gamma, 0, 0, 0)$ \emph{super-exponentially}.

Let
\begin{displaymath}
H_n(x,y) :=
\begin{bmatrix}
a_n(x,y) \\
y
\end{bmatrix}
\hspace{5mm} \text{and} \hspace{5mm}
V_n(x,y) :=
\begin{bmatrix}
x \\
\eta_n \circ \xi_n \circ ( p_2B_n)^{-1} (y)
\end{bmatrix}
\end{displaymath}
be the non-linear changes of coordinates given in \eqref{eq:nonlinear change of coord}, let
\begin{displaymath}
T_n(x,y) := (x+d_n, y),
\end{displaymath}
be the translation map given in \eqref{eq:crit projection1}, and let
\begin{displaymath}
s_n(x,y) := (\lambda_n x, \lambda_n y)
\hspace{5mm} , \hspace{5mm}
|\lambda_n| < 1
\end{displaymath}
be the scaling map so that if
\begin{equation} \label{eq:n change of coordinates}
\phi_n :=  H_n^{-1} \circ V_n^{-1} \circ T_n \circ s_n,
\end{equation}
then by proposition \ref{commuting}, we have
\begin{displaymath}
A_{n+1} = \phi_n^{-1} \circ A_n^{-1} \circ \Sigma_n^{\overline{\alpha}_4} \circ A_n \circ\phi_n
\end{displaymath}
and
\begin{displaymath}
B_{n+1} = \phi_n^{-1} \circ A_n^{-1} \circ \Sigma_n^{\overline{\alpha}_3} \circ A_n \circ \phi_n.
\end{displaymath}

Denote
\begin{displaymath}
\Phi^k_n := \phi_n \circ \phi_{n+1} \circ \ldots{} \circ \phi_{k-1} \circ \phi_k
\hspace{5mm} , \hspace{5mm}
\Omega^k_n := \Phi^k_n(\Omega)
\hspace{5mm} \text{and} \hspace{5mm}
\Gamma^k_n := \Phi^k_n(\Gamma).
\end{displaymath}
Define
\begin{displaymath}
U^k_n := \bigcup_{\overline{\omega} \prec \overline{\alpha}_{k-n}} \Sigma_n^{\overline{\omega}}(\Omega^k_n)
\hspace{5mm} \text{and} \hspace{5mm}
V^k_n := \bigcup_{\overline{\omega} \prec \overline{\alpha}_{k-n-1}} \Sigma_n^{\overline{\omega}}(\Gamma^k_n).
\end{displaymath}
It is not hard to see that $\{U^k_n \cup V^k_n\}^\infty_{k=n}$ form a nested sequence. Define the \emph{renormalization arc} of $\Sigma_n$ as 
\begin{equation} \label{siegel arc}
\gamma_n := \bigcap_{k=n}^\infty U^k_n \cup V^k_n.
\end{equation}

\begin{prop}
The renormalization arc $\gamma_n$ is invariant under the action of $\Sigma_n$. Moreover, if
\begin{displaymath}
p^k_n := \bigcup_{\overline{\omega} \prec \overline{\alpha}_{k-n}} \Sigma_n^{\overline{\omega}}(\Phi^k_n(\gamma_k \cap \Omega))
\hspace{5mm} \text{and} \hspace{5mm}
q^k_n := \bigcup_{\overline{\omega} \prec \overline{\alpha}_{k-n-1}} \Sigma_n^{\overline{\omega}}(\Phi^k_n(\gamma_k \cap \Gamma)),
\end{displaymath}
then
\begin{displaymath}
\gamma_n = p^k_n \cup q^k_n.
\end{displaymath}
\end{prop}

Let $\theta_*=(\sqrt{5}-1)/2$ be the golden mean rotation number, and let
\begin{displaymath}
I_L := [-\theta_*, 0]
\hspace{5mm} \text{and} \hspace{5mm}
I_R := [0, 1].
\end{displaymath}
Define $L : I_L \to \mathbb{R}$ and $R: I_R \to \mathbb{R}$ as
\begin{displaymath}
L(t) := t +1
\hspace{5mm} \text{and} \hspace{5mm}
R(t) := t - \theta_*.
\end{displaymath}
The pair $(R, L)$ represents rigid rotation of $\mathbb{R}/\mathbb{Z}$ by angle $\theta_*$.

The following is a classical result about the renormalization of 1D pairs.

\begin{prop}
Suppose $\|\Sigma\|_y = 0$. Then for every $n \geq 0$, there exists a quasi-symmetric homeomorphism between $I_L \cup I_R$ and the renormalization arc $\gamma_n$ that conjugates the action of $\Sigma_n=(A_n, B_n)$ and the action of $(R, L)$. Moreover, the renormalization arc $\gamma_n$ contains the unique critical point $c_n = 0$ of $\eta_n$.
\end{prop}

The following is shown in \cite{GaRY}.

\begin{thm} \label{thm:gry2}
Let $\Sigma = (A, B)$ be a commuting pair contained in the stable manifold $W^s(\iota(\zeta_*))$ of the 2D renormalization fixed point $\iota(\zeta_*)$. Then for every $n \geq 0$, there exists a homeomorphism between $I_L \cup I_R$ and the renormalization arc $\gamma_n$ that conjugates the action of $\Sigma_n=(A_n, B_n)$ and the action of $(R, L)$. Moreover, this conjugacy cannot be $C^1$ smooth.
\end{thm}

Theorem \ref{thm:gry} follows from the above statement and the following:

\begin{thm}[\cite{GaRY}]
Suppose
\begin{displaymath}
\Sigma = \Sigma_{H_{\mu_*,\nu}},
\end{displaymath}
where $\Sigma_{H_{\mu_*,\nu}}$ is the renormalization of the H\'enon map given in theorem \ref{thm:henon}. Then the linear rescaling of the renormalization arc $s_0(\gamma_0)$ is contained in the boundary of the Siegel disc $\Delta$ of $H_{\mu_*, \nu}$. In fact, we have
\begin{displaymath}
\partial \Delta = s_0(\gamma_0) \cup H_{\mu_*, \nu} \circ s_0(\gamma_0).
\end{displaymath}
\end{thm}

Henceforth, we consider the renormalization arc of $\Sigma_n$ as a continuous curve $\gamma_n= \gamma_n(t)$ parameterized by $I_L \cup I_R$. The components of $\gamma_n$ are denoted
\begin{displaymath}
\gamma_n(t) =
\begin{bmatrix}
\gamma^x_n(t) \\
\gamma^y_n(t)
\end{bmatrix}.
\end{displaymath}
Lastly, denote the renormalization arc of $\iota(\zeta_*)$ by
\begin{displaymath}
\gamma_*(t) =
\begin{bmatrix}
\gamma^x_*(t) \\
\gamma^y_*(t)
\end{bmatrix}.
\end{displaymath}

The following are consequences of Theorem \ref{thm:hyperb2}.

\begin{cor} \label{convergence}
As $n \to \infty$, we have the following convergences (each of which occurs at a geometric rate):
\begin{enumerate}
\item $\eta_n \to \eta_*$,
\item $\lambda_n \to \lambda_*$ (hence $s_n \to s_*$),
\item $\phi_n \to \psi_*$, where
\begin{displaymath}
\psi_*(x,y) =
\begin{bmatrix}
\eta_*^{-1}(\lambda_* x) \\
\eta_*^{-1}(\lambda_* y)
\end{bmatrix},
\hspace{5mm} \text{and}
\end{displaymath}
\item $\gamma_n \to \gamma_*$ (hence $|\gamma^x_n(0)| \to 0$).
\end{enumerate}
\end{cor}

\subsection{Normality of the compositions of scope maps}

\noindent
Define
\begin{displaymath}
\psi_n(x,y) :=
\begin{bmatrix}
\eta_n^{-1}(\lambda_n x) \\
\eta_n^{-1}(\lambda_n y)
\end{bmatrix}.
\end{displaymath}
For $n \leq k$, denote
\begin{displaymath}
\Psi^k_n := \psi_n \circ \psi_{n+1} \circ \ldots{} \circ \psi_{k-1} \circ \psi_k.
\end{displaymath}

Let
\begin{displaymath}
\begin{bmatrix}
\sigma^k_n & 0 \\
0 & \sigma^k_n
\end{bmatrix}:=
(D_{(0,0)} \Psi^k_n)^{-1}.
\end{displaymath}

\begin{prop} \label{normality of tilde}
The family $\{\sigma^k_n \Psi^k_n\}^\infty_{k=n}$ is normal.
\end{prop}

\begin{proof}
By corollary \ref{convergence}, there exists a domain $U \subset \mathbb{C}^2$ and a uniform constant $c<1$ such that for all $k$ sufficiently large, the map $\psi_k$ is well defined on $U$, and
\begin{displaymath}
\Omega \cup A_{k+1}(\Omega) \cup \Gamma \cup B_{k+1}(\Gamma) \Subset c U.
\end{displaymath}
Thus, by choosing a smaller domain $U$ if necessary, we can assume that $\psi_k$ and hence, $\Psi^k_n$ extends to a strictly larger domain $V \Supset U$. It follows from applying Ko\'ebe distortion theorem to the first and second coordinate that $\{\sigma^k_n \Psi^k_n\}^\infty_{k=n}$ is a normal family.
\end{proof}

\begin{prop}\label{diagonal scope}
There exists a uniform constant $M >0$ such that
\begin{displaymath}
||\phi_n-\psi_n|| < M \epsilon^{2^{n-1}}.
\end{displaymath}
\end{prop}

\begin{proof}
The result follows readily from \eqref{shadowing} and \eqref{eq:error estimate}.
\end{proof}

\begin{prop} \label{diagonal comp scope}
There exists a uniform constant $K>0$ such that
\begin{displaymath}
\sigma^k_n||\Phi^k_n-\Psi^k_n|| < K \epsilon^{2^{n-1}}.
\end{displaymath}
\end{prop}

\begin{proof}
By proposition \ref{diagonal scope}, we have
\begin{displaymath}
\phi_{k-1} = \psi_{k-1} + \tilde{E}_{k-1} \hspace{5mm} \text{and} \hspace{5mm} \phi_k = \psi_k + E_k,
\end{displaymath}
where $||\tilde{E}_{k-1}||<M \epsilon^{2^{k-2}}$ and $||E_k||<M\epsilon^{2^{k-1}}$. Observe that
\begin{align*}
\phi_{k-1} \circ \phi_k &= \phi_{k-1} \circ (\psi_k + E_k) \\
&= \phi_{k-1} \circ \psi_k + \bar{E}_k \\
&= (\psi_{k-1} + \tilde{E}_{k-1}) \circ \psi_k + \bar{E}_k \\
&= \psi_{k-1} \circ \psi_k + \tilde{E}_{k-1} \circ \psi_k + \bar{E}_k,
\end{align*}
where $||\bar{E}_k||< L \epsilon^{2^{k-1}}$ for some uniform constant $L>0$ by corollary \ref{convergence}. Let
\begin{displaymath}
E_{k-1} := \tilde{E}_{k-1} + \bar{E}_k \circ \psi_k^{-1}.
\end{displaymath}
By corollary \ref{convergence}, $\psi_k^{-1}$ is uniformly bounded, and hence, we have
\begin{displaymath}
||E_{k-1}|| < M \epsilon^{2^{k-2}} + 2L\epsilon^{2^{k-1}} < 2M \epsilon^{2^{k-2}}.
\end{displaymath}
Thus, we have
\begin{displaymath}
\phi_{k-1} \circ \phi_k = \psi_{k-1} \circ \psi_k + E_{k-1} \circ \psi_k.
\end{displaymath}
Proceeding by induction, we obtain
\begin{displaymath}
\Phi^k_n = \Psi^k_n + E_n \circ \psi_{n+1} \circ \ldots{} \circ \psi_k,
\end{displaymath}
where
\begin{displaymath}
||E_n|| < 2M \epsilon^{2^{n-1}}.
\end{displaymath}

By definition, we have
\begin{displaymath}
\sigma^k_n (\psi_n \circ \psi_{n+1} \circ \ldots{} \circ \psi_k)'(0) = 1.
\end{displaymath}
Factor the scaling constant as
\begin{displaymath}
\sigma^k_n := \dot{\sigma}^k_n \sigma^k_{n+1},
\end{displaymath}
so that
\begin{displaymath}
|\dot{\sigma}^k_n \psi_n'(\psi_{n+1}\circ \ldots{} \circ \psi_k(0))| = 1,
\end{displaymath}
and
\begin{displaymath}
|\sigma^k_{n+1} (\psi_{n+1} \circ \ldots{} \circ \psi_k)'(0)| = 1.
\end{displaymath}
Let
\begin{displaymath}
M := \sup_{x \in Z}  \eta_n'(x).
\end{displaymath}
Observe that $\dot{\sigma}^k_n$ is uniformly bounded by $\lambda_n^{-1}M$. Moreover, by proposition \ref{normality of tilde}, we have that $\sigma^k_{n+1} (\psi_{n+1} \circ \ldots{} \circ \psi_k)'$ is also uniformly bounded. Therefore,
\begin{align*}
||\sigma^k_n  (E_n \circ \psi_{n+1} \ldots{} \circ \psi_n)'|| &= ||\dot{\sigma}^k_n E_n'(\psi_{n+1} \ldots{} \circ \psi_n)|| \cdot ||\sigma^k_{n+1} (\psi_{n+1} \ldots{} \circ \psi_n)'|| \\
&= K ||E_n'(\psi_{n+1} \ldots{} \circ \psi_n)|| \\
&< K \epsilon^{2^{n-1}}
\end{align*}
for some universal constant $K>0$.
\end{proof}

By proposition \ref{normality of tilde} and \ref{diagonal comp scope}, we have the following theorem.

\begin{thm} \label{normality}
The family $\{\sigma^k_n \Phi^k_n\}_{k=n}^\infty$ is normal.
\end{thm}

\subsection{The boundary of the Siegel disk is not smooth.}

\subsection{The boundary of the Siegel disk is not smooth.}

Let $[t_l, t_r] \subset \mathbb{R}$ be a closed interval, and let $C : [t_l, t_r] \to \mathbb{C}$ be a smooth curve. For any subset $N \subset \mathbb{C}$ intersecting the curve $C$, we define the {\it angular deviation of $C$ on $N$} as
\begin{equation} \label{angular dev}
\angdev(C, N) := \sup_{t, s \in C^{-1}(N)}|\arg(C'(t))- \arg(C'(s))|,
\end{equation}
where the function arg$: \mathbb{C} \to \mathbb{R} / \mathbb{Z}$ is defined as
\begin{equation} \label{arg}
\arg(re^{2\pi \theta i}) := \theta.
\end{equation}

\begin{lem} \label{straight lem}
Let $\theta \in \mathbb{R} /\mathbb{Z}$, and let $C_\theta : [0, 1] \to \mathbb{C}$ be a smooth curve such that $C_\theta(0) =0$ and $C_\theta(1) = e^{2\pi \theta i}$. Then for some $t \in [0, 1]$, we have
\begin{displaymath}
\arg(C_\theta'(t)) = \theta.
\end{displaymath}
\end{lem}

\begin{lem} \label{corner lem}
Let
\begin{equation} \label{square}
q_2(x) := x^2
\hspace{5mm} \text{and} \hspace{5mm}
A^R_r := \{z \in \mathbb{C} \, | \, r < |z| < R\}.
\end{equation}
Suppose $C : [t_l,t_r] \to \mathbb{D}_R$ is a smooth curve such that $|C(t_l)|=|C(t_r)|=R$, and $|C(t_0)| < r$ for some $t_0 \in [t_l, t_r]$. Then for every $\delta >0$, there exists $M >0$ such that if $\emph{mod}(A^R_r) > M$, then either $\angdev(C, \mathbb{D}_R)$ or $\angdev(q_2 \circ C, \mathbb{D}_{R^2})$ is greater than $1/6-\delta$.
\end{lem}

\begin{proof}
Without loss of generality, assume that $R = 1$, and $C(t_r) = 1$. We prove the case when $r=0$, so that $C(t_0) = 0$. The general case follows by continuity.

Suppose that $\angdev(C, \mathbb{D}_R) < 1/6$. Then by lemma \ref{straight lem}, we have
\begin{displaymath}
1/3 < \arg(C(t_l)) < 2/3.
\end{displaymath}
This implies that
\begin{displaymath}
-1/3 < 2\arg(C(t_1)) < 1/3.
\end{displaymath}
Hence, by lemma \ref{straight lem}, we have $\angdev(q_2 \circ C, \mathbb{D}_{R^2}) > 1/6$.
\end{proof}

\begin{cor} \label{deformed corner lem}
Let $W \subset \mathbb{C}$ be a simply connected neighbourhood of $0$, let $C : [t_l, t_r] \to \overline{W}$ and $E: [t_l, t_r] \to \mathbb{C}$ be smooth curves, and let $f : W \to \mathbb{C}$ be a holomorphic function with a unique simple critical point at $c \in \mathbb{D}_r$ for $r < 1$. Consider the smooth curve
\begin{displaymath}
\tilde{C} := f \circ C + E.
\end{displaymath}
Suppose $C(t_l), C(t_r) \in \partial W$, and $|C(t_0)| < r$ for some $t_0 \in [t_l, t_r]$. Then for every $\delta >0$, there exists $\epsilon >0$ and $M >0$ such that if $\|E\|< \epsilon$ and $\emph{mod}(W \setminus \mathbb{D}_r) > M$, then either $\angdev(C, W)$ or $\angdev(\tilde{C}, f(W))$ is greater than $1/6-\delta$.
\end{cor}

Let $U \subset Z \subset \mathbb{C}$ be a simply-connected domain containing the origin. For all $k$ sufficiently large, the unique critical point $c_k$ of $\eta_k$ is contained in $U$. Let $V_k := \eta_k(U)$. Then there exists conformal maps $u_k : (\mathbb{D}, 0) \to (U, c_k)$ and $v_k : (\mathbb{D}, 0) \to (V_k, \eta_k(c_k))$ such that the following diagram commutes:
\comdia{\mathbb{D}}{U}{\mathbb{D}}{V_k}{u_k}{v_k}{q_2}{\eta_k}

By corollary \ref{convergence}, we have the following result:

\begin{prop} \label{square converge}
The maps $u_k : (\mathbb{D}, 0) \to (U, c_k)$ and $v_k : (\mathbb{D}, 0) \to (V_k, \eta_k(c_k))$ converge to conformal maps $u_* : (\mathbb{D}, 0) \to (U, 0)$ and $v_* : (\mathbb{D}, 0) \to (\eta_*(U), \eta_*(0))$. Moreover, the following diagram commutes:
\comdia{\mathbb{D}}{U}{\mathbb{D}}{\eta_*(U)}{u_*}{v_*}{q_2}{\eta_*}
\end{prop}

\begin{proof}[Proof of Non-smoothness]
By theorem \ref{normality}, the sequence $\{\sigma^k_0 \Phi^k_0\}_{k=0}^\infty$ has a converging subsequence. By replacing the sequence by this subsequence if necessary, assume that $\{\sigma^k_0 \Phi^k_0\}_{k=0}^\infty$ converges. Consider the following commutative diagrams:
\begin{displaymath}
\begin{CD}
\mathbb{D} @> u_k >> U\\
@VV q_2 V @VV \eta_k V\\
\mathbb{D} @> v_k >> V_k
\end{CD}
\hspace{8mm} \text{and} \hspace{5mm}
\begin{CD}
\Omega @> \Phi^k_0 >> \Omega\\
@VV A_k V @VV A_0 V\\
A_k(\Omega) @> \Phi^k_0 >> A_0(\Omega)
\end{CD}.
\end{displaymath}

Let $\delta >0$. Since $\{\sigma^k_0 \Phi^k_0\}_{k=0}^\infty$ converges, we can choose $R>0$ sufficiently small so that if
\begin{displaymath}
X_k := u_k(\mathbb{D}_R) \subset U_k,
\hspace{5mm} \text{and} \hspace{5mm}
Y_k := v_k (\mathbb{D}_{R^2}) \subset V_k,
\end{displaymath}
then for any smooth curves $C_1 \subset \Omega := U \times U$ and $C_2 \subset A_k(\Omega)$ intersecting $X_k \times X_k$ and $Y_k \times Y_k$ respectively, we have
\begin{displaymath}
\kappa \angdev(C_1, X_k \times X_k) < \angdev(\Phi^k_0 \circ C_1, \Phi^k_0(X_k \times X_k))
\end{displaymath}
and
\begin{displaymath}
\kappa \angdev(C_2, Y_k \times Y_k) < \angdev(\Phi^k_0 \circ C_2, \Phi^k_0(Y_k \times Y_k))
\end{displaymath}
for some uniform constant $\kappa>0$.

Consider the renormalization arc of $\Sigma_n$:
\begin{displaymath}
\gamma_n(t) =
\begin{bmatrix}
\gamma^x_n(t) \\
\gamma^y_n(t)
\end{bmatrix}.
\end{displaymath}
Recall that we have
\begin{displaymath}
\gamma^x_n, \gamma^y_n \to \gamma_*
\hspace{5mm} \text{as} \hspace{5mm}
n \to \infty,
\end{displaymath}
where $\gamma_*$ is the renormalization arc of the 1D renormalization fixed point $\zeta_*$. Now, choose $r>0$ is sufficiently small so that the annulus $X_k \setminus \mathbb{D}_r$ satisfies the condition of Corollary \ref{deformed corner lem} for the given $\delta$. Next, choose $K$ sufficiently large so that for all $k > K$, we have
\begin{displaymath}
|c_k|, |\gamma^x_k(0)| < r,
\end{displaymath}
and 
\begin{displaymath}
A_k = \begin{bmatrix}
a_k\\
h_k
\end{bmatrix}
=
\iota(\eta_k) + (e_x, e_y),
\end{displaymath}
such that $\|E := (e_x, e_y)\| < \epsilon$, where $\epsilon$ is given in Corollary \ref{deformed corner lem}.

Now, suppose towards a contradiction that the renormalization arc $\gamma_0$ of $\Sigma_0$, and hence the renormalization arc $\gamma_k$ of $\Sigma_k$ for all $k \geq 0$, are smooth. By the above estimates, we can conclude:
\begin{align*}
\angdev(\gamma_0, \Phi^k_0(X_k \times X_k)) &= \angdev(\Phi^k_0 \circ \gamma_k, \Phi^k_0(X_k \times X_k))\\
&> \kappa \angdev(\gamma_k, X_k \times X_k)\\
&> \kappa \angdev(\gamma^x_k, X_k)
\end{align*}
and
\begin{align*}
\angdev(\gamma_0, \Phi^k_0(Y_k \times Y_k)) &= \angdev(\Phi^k_0 \circ \gamma_k, \Phi^k_0(Y_k \times Y_k))\\
&> \kappa \angdev(\gamma_k, Y_k \times Y_k)\\
&= \kappa \angdev(A_k \circ \gamma_k, Y_k \times Y_k)\\
&> \kappa \angdev(a_k \circ \gamma_k, Y_k)\\
&= \kappa \angdev(\eta_k \circ \gamma^x_k + e_x(\gamma_k), Y_k).
\end{align*}

By Lemma \ref{deformed corner lem}, either $\angdev(\gamma^x_k, X_k)$ or $\angdev(a_k \circ \gamma_k, Y_k)$ is greater than $1/6 - \delta$. Hence,
\begin{displaymath}
\max\{\angdev(\gamma_0, \Phi^k_0(X_k \times X_k)), \angdev(\gamma_0, \Phi^k_0(Y_k \times Y_k))\} > l
\end{displaymath}
for some uniform constant $l > 0$. Since $\Phi^k_0(X_k \times X_k)$ and $\Phi^k_0(Y_k \times Y_k)$ both converge to a point in $\gamma_0$ as $k \to \infty$, this is a contradiction.
\end{proof}

\end{document}